\documentclass[12pt]{article}

\usepackage{amssymb}
\usepackage{array}
\usepackage{latexsym}
\usepackage{enumerate}
\usepackage{amsmath}
\usepackage{amsfonts}
\usepackage{amsthm}
\usepackage{graphicx}
\usepackage{comment}
\usepackage[english]{babel}
\usepackage[T1]{fontenc}
\usepackage{geometry}
\usepackage{color}
\usepackage{hyperref}
\ifx\smallsetminus\undefined
\def\smallsetminus{\setminus}
\fi

\def\titlerunning#1{\gdef\titrun{#1}}
\makeatletter
\def\author#1{\gdef\autrun{\def\and{\unskip, }#1}\gdef\@author{#1}}

\makeatother
\def\email#1{\hspace*{4pt}{\em e-mail}: #1}

\date{}
\theoremstyle{plain}
\newtheorem{prop}{Proposition}[section]
\newtheorem{theorem}[prop]{Theorem}

\newtheorem{lemma}[prop]{Lemma}
\theoremstyle{definition}
\newtheorem{remark}[prop]{Remark}
\newtheorem{conjecture}[prop]{Conjecture}

\newcommand{\cG}{{\mathcal G}}

\newcommand{\cB}{\mathcal B}
\newcommand{\cE}{\mathcal E}

\newcommand{\cP}{\mathcal P}

\newcommand{\cF}{\mathcal F}

\newcommand{\cH}{\mathcal H}
\newcommand{\cM}{\mathcal M}
\newcommand{\fG}{\mathfrak G}

\newcommand{\PG}{\mathrm{PG}}
\newcommand{\PGamma}{\mathrm{P\Gamma}}
\newcommand{\AG}{\mathrm{AG}}
\newcommand{\GF}{\mathrm{GF}}

\newcommand{\cV}{\mathcal V}

\begin{document}

\titlerunning{}

\title{On the equivalence of certain quasi-Hermitian varieties}

\author{Angela Aguglia\footnote{Dipartimento di Meccanica, Matematica e Management, Politecnico di Bari, Via Orabona 4, I-70125 Bari, Italy; \email{angela.aguglia@poliba.it}}
\and
Luca Giuzzi\footnote{DICATAM, University of Brescia, Via Branze 53,
  I-25123 Brescia, Italy; \email{luca.giuzzi@unibs.it}}}

\maketitle

\begin{abstract}
In \cite{ACK} new quasi-Hermitian varieties $\cM_{\alpha,\beta}$
in $\PG(r,q^2)$ depending on a pair of parameters $\alpha,\beta$
from the underlying field $\GF(q^2)$ have been constructed.
In the present paper we study the structure of
the lines contained in $\cM_{\alpha,\beta}$ and consequently
determine the number of inequivalent such varieties for $q$ odd and $r=3$. As a byproduct, we also prove that the collinearity graph of $\cM_{\alpha,\beta}$ is connected  with diameter $3$ for $q\equiv 1\pmod 4$.
\end{abstract}
\section{Introduction}
It is a well-known problem in finite geometry to characterize
the absolute points of a polarity in terms of their combinatorial
properties.
In this line of investigation, one of the most celebrated results
is the Segre's Theorem stating that  in a Desarguesian
projective plane $\PG(2,q)$ of odd order $q$ a set $\Omega$ which has
the same number of points, namely $q+1$, and the same intersections
with lines as a conic (i.e. $0$, $1$ or $2$) is indeed a conic;
see~\cite{S55}.

As the dimension grows,
the combinatorics of the intersection with subspaces turns out not to
be enough as to characterize the absolute points of a polarity  in
the orthogonal case as well as in the unitary one.
%; in any case  however
%they restrict the possibilities and usually lead to interesting
%objects. For instance, see~\cite[Chapter 5]{U11} and the references
%therein for the notion of \emph{quadratic sets}.

The set of the absolute points of a Hermitian polarity  of $\PG(r,q^2)$ is a {\em non-singular Hermitian variety}.
%even in the plane case $\PG(2,q^2)$; indeed, non-classical unitals are well known
%and widely studied; see~\cite{BE}.

Quasi-Hermitian varieties of $\PG(r,q^2)$ are a generalization of
non-singular Hermitian varieties as they are defined as follows.
Let $q$ be any prime power and assume $r\geq 2$;
a \emph{quasi-Hermitian variety} of $\PG(r,q^2)$ is a
set of points having the same size and the same intersection numbers with hyperplanes as a non-singular Hermitian variety $\cH(r,q^2)$.
In particular, the intersection numbers with hyperplanes of $\cH(r,q^2)$ only take two values thus,  quasi-Hermitian varieties are  two-character sets; see \cite{CK,De} for an overview of their applications.
The Hermitian variety $\cH(r,q^2)$ can be viewed trivially as a quasi-Hermitian variety; as such it is called the \emph{classical
quasi-Hermitian variety of $\PG(r,q^2)$}.

For $r=2$, a quasi-Hermitian variety of $\PG(2,q^2)$ is also called a
\emph{unital} or \emph{Hermitian arc}. Non-classical unitals have been extensively studied and
characterized~\cite{BE} and some constructions are known;
see for instance~\cite{agk10}.
As far as we know, the only known non-classical quasi-Hermitian
varieties of $\PG(r,q^2)$, $r\geq3$ were constructed in
\cite{AA,ACK,DS,FP} and they are not isomorphic
among themselves; see \cite{FP}.

In \cite{ACK},  quasi-Hermitian varieties $\cM_{\alpha,\beta}$
of $\PG(r,q^2)$ with $r\geq 2$, depending on a pair of parameters $\alpha,\beta$
from the underlying field $\GF(q^2)$, were constructed.
For $r=2$ these varieties
are Buekenhout-Metz (BM) unitals, see \cite{BE,BE92,GE0}.
%In the present paper we shall show that also the equivalence classes
%up to projectivities of the varieties $\cM_{\alpha,\beta}$ for $r=3$
%have the same structure as those of BM unitals (and we conjecture this
%to be the case for all values of $r$).
As such, for $r\geq 3$
we shall call $\cM_{\alpha,\beta}$ the \emph{BM quasi-Hermitian variety} of parameters $\alpha$ and $\beta$ of $\PG(r,q^2)$.

The number of  inequivalent BM unitals in $\PG(2,q^2)$ has been computed in \cite{BE92} for $q$ odd and in \cite{GE0} for $q$ even. In the present paper we shall enumerate the BM quasi-Hermitian varieties in $\PG(3,q^2)$ with $q$ odd and show that they behave under this respect in a similar way as BM unitals in $\PG(2,q^2)$. Our long-term aim is to
try to find a characterization of the BM quasi-Hermitian varieties among all possible quasi-Hermitian varieties in spaces of the same dimension and
order.

Apart from the Introduction, the paper is organized into 4 sections. In Section~\ref{sec:1} we describe the construction of the BM quasi-Hermitian varieties in $\PG(3,q^2)$ whereas in Section~\ref{sec:2}   we determine the number  of   lines of $\PG(3,q^2)$, $q$ odd, through a point of $\cM_{\alpha,\beta}$  which are entirely contained in $\cM_{\alpha,\beta}$.
By using this result in Section~\ref{sec:3}, we prove that the collinearity graph of $\cM_{\alpha,\beta}$ is connected for $q\equiv 1\pmod 4$ (which is the only interesting case, as for $q\equiv 3\pmod 4$
the only lines contained in $\cM_{\alpha,\beta}$ are those of a pencil
of $(q+1)$-lines, all contained in a plane).
Finally, in Section~\ref{sec:4}, we prove  our main result:
\begin{theorem}
  \label{main-th}
  Let $q=p^n$ with $p$ an odd prime. Then
the number $N$ of inequivalent quasi-Hermitian varieties $\cM_{\alpha,\beta}$ of $\PG(3,q^2)$ is
\[N=\frac{1}{n}\left(\sum_{k|n}\Phi\left(\frac{n}{k}\right)p^k \right)-2,\]
where $\Phi$ is the Euler $\Phi$-function.
\end{theorem}
As a byproduct of our arguments, we also obtain a simple way to determine
when two quasi-Hermitian varieties are equivalent, see Lemma~\ref{lemadd5} and
Lemma~\ref{main-lemma} for the details.

\section{Preliminaries}
\label{sec:1}

In this section we recall the construction of
the  BM quasi-Hermitian varieties $\cM_{\alpha,\beta}$ of $\PG(3,q^2)$ described in \cite{ACK}.

Fix a projective frame in $\PG(3,q^2)$ with homogeneous coordinates
$(J,X,Y,Z)$, and consider the affine space $\AG(3,q^2)$
with infinite hyperplane $\Sigma_{\infty}$ of $J=0$.
Then, the affine coordinates for points of  $\AG(3,q^2)$ are denoted by $(x,y,z)$,
where $x=X/J$, $y=Y/J$ and $z=Z/J$.
Set
\[
  %\begin{equation}
%\label{heinf}
\cF=\{(0,X,Y,Z)\colon X^{q+1}+Y^{q+1}=0\};
%\end{equation}
\]
this can be viewed as a Hermitian cone of $\Sigma_{\infty}\cong\PG(2,q^2)$ projecting a Hermitian variety of $\PG(1,q^2)$.
Now take $\alpha \in\GF(q^2)^*$ and $\beta\in\GF(q^2)\setminus\GF(q)$
and consider the  algebraic variety $\cB_{\alpha,\beta}$
of projective equation
\begin{equation}\label{eqqh}
\cB_{\alpha,\beta}: \  Z^q J^q-ZJ^{2q-1}+\alpha^q(X^{2q}+Y^{2q})-\alpha(X^2+Y^2)J^{2q-2}=(\beta^q-\beta)(X^{q+1}+Y^{q+1})J^{q-1}.
\end{equation}

We observe that
\begin{itemize}
\item $\cB_{\infty}:=\cB_{\alpha,\beta}\cap \Sigma_{\infty}$ is the union of two lines $\ell_1: X-\nu Y=0=J$ and $\ell_2: X+\nu Y=0=J$,
 with $\nu \in \GF(q^2)$ such that $\nu^2+1=0$ if $q$ is odd;
\item Let $P_{\infty}:=(0,0,0,1)$.  Then,  $\ell_1 \cap \ell_2=P_{\infty}$;
\item $ \cB_{\infty} \subseteq \cF$ if $q\equiv1\pmod4$ or $q$ is even.
\end{itemize}

It is shown in~\cite{ACK} that the point set
\begin{equation}
  \label{qmv}
  \cM_{\alpha,\beta}:= (\cB_{\alpha,\beta}\setminus \Sigma_{\infty})\cup \cF,
\end{equation}
that is, the union of the affine
points of $\cB_{\alpha,\beta}$ and $\cF$, is
a quasi-Hermitian variety  of $\PG(3,q^2)$
for any  $q>2$ even or
for $q$ odd  and $4\alpha^{q+1}+(\beta^q-\beta)^2 \neq 0$.
This is the variety we shall consider in the present paper
in the case in which $q$ is odd.

We stress that \eqref{eqqh} is not the equation of $\cM_{\alpha,\beta}$.
However, any set of points in a finite projective space can be
endowed of the structure of an algebraic variety, so we shall
speak of the \emph{variety} $\cM_{\alpha,\beta}$ even if we do
not provide an equation for it.

\section{Combinatorial properties of $\cM_{\alpha,\beta}$}
% \subsection{Line sections of $\cM_{\alpha,\beta}$}
\label{sec:2}
%In \cite{-ACG} the spectrum of all possible intersection numbers
%between the varieties $\cB_{\alpha,\beta}$ and   lines   of $\PG(r,q^2)$ was
%provided. In particular, for $r=3$ we have the following.
%\begin{theorem}\label{lines}
%Let $\ell$ be a line of $\PG(3, q^2)$. Then, the possible intersection sizes for $\ell \cap \cB_{\alpha,\beta}$ are as follows
%\[0,1,2,q-1, q,q+1,q+2, 2q-1,2q,q^2+1 \]
%\end{theorem}

We first determine the number of lines passing through each point of $\cM_{\alpha,\beta}$ of $\PG(3,q^2)$, for $q$  odd.
 We recall the following (see~\cite[Corollary 1.24]{HJ}).
\begin{lemma}\label{tec}
 Let $q$ be an odd prime power. The equation
%  \begin{equation}
%    \label{x0}
\[
  X^q+aX+b=0
\]
%  \end{equation}
  admits  exactly one solution in $\GF(q^2)$  if and only if
  $a^{q+1}\neq1$.
  When $a^{q+1}=1$, the aforementioned equation has either $q$ solutions
  when $b^q=a^qb$ or no solution when $b^q\neq a^qb$.
  \end{lemma}
%\begin{proof}
 % We have $X^q=-aX-b$, whence
  %\begin{equation}
  %  \label{x1}
   % X=-a^qX^q-b^q
  %\end{equation}
  %Replacing $X$ in~\eqref{x0} we obtain
  %\begin{equation}
  %  X^q-a(a^qX^q+b^q)+b=0,
  %\end{equation}
  %whence
  %\begin{equation}
   % \label{x3}
   % X^q(1-a^{q+1})=ab^q-b.
  %\end{equation}
  %which has exactly one solution if and only if $a^{q+1}\neq 1$.
  %Suppose now $a^{q+1}=1$ and that~\eqref{x0} admits solution.
  %If $b=0$, then $X=0$ is a solution
  %of~\eqref{x0}. Also, any $X\neq 0$ such that $X^{q-1}=-a$ is a solution.
  %On the other hand $a^{q+1}=1$ implies $a=c^{q-1}$ for some $c\in\GF(q^2)$.
  %Also in $\GF(q^2)$ the element $-1$ can always be written as $-1=d^{q-1}$.
  %It follows that~\eqref{x0} admits $q-1$ non-zero solutions.
  %We now consider the non-homogeneous case; it follows from elementary
  %linear algebra that and solution of~\eqref{x0} with $b\neq0$ can
  %be obtained by adding to a special solution $X_0$ a solution of
  %$X^q+aX=0$. So, the solutions for $b\neq0$ are either $0$ or $q$.
  %Clearly,~\eqref{x3} must be satisfied; so $ab^q=b$ is a necessary
  %condition for~\eqref{x0} to have solution in this case. Put
  %$X_0=-2^{-1}b^q$. Then we get $-2^{-1}b+{-2}^{-1}ab^q+b=0$ which is
  %satisfied. So, in this case~\eqref{x0} admits at least a solution, hence
  %$q$.
 %\end{proof}

\begin{lemma}
  \label{lm32}
 Let $\cB_{\alpha,\beta}$ be the projective variety of Equation~\eqref{eqqh} and $\cB_{\infty}$ be
the intersection of the variety $\cB_{\alpha,\beta}$ with the
hyperplane at infinity $\Sigma_{\infty}:J=0$ of $\PG(3,q^2)$.
  \begin{itemize}
  \item
    If  $q\equiv 1\pmod 4$,
    then, for any affine point $L$ of $\cB_{\alpha,\beta}$
    there are exactly $2$ lines contained in $\cB_{\alpha,\beta}$ through $L$;
    for any  point $L_{\infty}\in \cB_{\infty}$ with
    $L_{\infty}\neq P_{\infty}$,
    there are $q+1$ lines  of a pencil through
    $L_{\infty}$ contained in $\cB_{\alpha,\beta}$.
    If  $q\equiv 3\pmod 4$ then no line of $\cB_{\alpha,\beta}$ passes through any affine point of $\cB_{\alpha,\beta}$ whereas through a point at infinity of $\cB_{\infty}$ different from $P_{\infty}$
    there pass only one line contained in $\cB_{\alpha,\beta}$.
  \item
    There are exactly two lines of $\cB_{\alpha,\beta}$ through $P_{\infty}$ for all odd $q$.
\end{itemize}
  \end{lemma}

\begin{proof} Let $\ell$ be  a line   of $\PG(3,q^2)$ passing through an affine point of $\cB_{\alpha,\beta}$.
The affine points $P(x,y,z)$ of $\cB_{\alpha,\beta}$ satisfy the equation:
%\begin{equation}\label{eqqhaff}
\[
\cB_{\alpha,\beta}: \  z^q-z+\alpha^q(x^{2q}+y^{2q})-\alpha(x^2+y^2)=(\beta^q-\beta)(x^{q+1}+y^{q+1}).
\]
% \end{equation}
 From~\cite[\S 4]{ACK}, it can be directly
  seen that the collineation group
  of $\cB_{\alpha,\beta}$ acts transitively on its affine points.
  Thus, we can assume that $\ell$ passes through the origin $O=(1,0,0,0)$
  of the fixed frame and hence it has affine parametric equations:
  \[
    \left\{\begin{array}{l}
       x= m_1 t \\
      y= m_2 t \\
     z= m_3 t
      \end{array}\right.
  \]
  with $t$ ranging over $\GF(q^2)$.
We study the following system
  \begin{small}
\begin{equation}
    \left\{\begin{array}{l}\label{sis1}
      z^q-z+\alpha^q(x^{2q}+y^{2q})-\alpha(x^2+y^2)=(\beta^q-\beta)(x^{q+1}+y^{q+1})\\
       x=m_1t\\
      y=m_2t\\
     z=m_3t
      \end{array}\right.
  \end{equation}
  \end{small}
  As proved in \cite[Theorem 4.3]{ACG}  $\ell$ can be contained in $\cB_{\alpha,\beta}$  only if $m_3=0$.
  Thus assume $m_3=0$ and replace the parametric
  values of $(x,y,z)$ in the first equation of~\eqref{sis1}.
  We obtain that
  \begin{equation}
    \label{eq4}
    (t^{2}\alpha (m_1^{2}+m_2^2))^q-t^2\alpha(m_1^2+m_2^2)=
    t^{q+1}(\beta^q-\beta)(m_1^{q+1}+m_2^{q+1})
  \end{equation}
  must hold for all $t\in\GF(q^2)$. Considering separately the cases $t\in\GF(q)$
  and $t=\lambda$ with $\lambda\in\GF(q^2)\setminus\GF(q)$ we obtain the following systems, $\forall\lambda\in\GF(q^2)\setminus\GF(q)$
  \[
    \begin{cases}    (\alpha^q(m_1^2+m_2^2)^q-\alpha(m_1^2+m_2^2))=(\beta^q-\beta)(m_1^{q+1}+m_2^{q+1})
      \\
      \lambda^{2q}\alpha^q(m_1^2+m_2^2)^q-\lambda^2\alpha(m_1^2+m_2^2)=
      \lambda^{q+1}(\beta^q-\beta)(m_1^{q+1}+m_2^{q+1}). \
    \end{cases}
  \]
  Replacing the first equation in the second, we get
    \[ \forall\lambda\in\GF(q^2)\setminus\GF(q): \lambda^{2q}\alpha^q(m_1^2+m_2^2)^q(1-\lambda^{1-q})=
      \lambda^{2}\alpha(m_1^2+m_2^2)(1-\lambda^{q-1}).
    \]
   Observe that $(1-\lambda^{1-q})=\frac{\lambda^{q-1}-1}{\lambda^{q-1}}$. Suppose $m_1^2+m_2^2\neq 0$. Then,
    \[ \lambda^{2q-2}\alpha^{q-1}(m_1^2+m_2^2)^{q-1}=-\lambda^{q-1}, \]
   whence $(\lambda\alpha(m_1^2+m_2^2))^{q-1}=-1$ for all $\lambda\in\GF(q^2)\setminus\GF(q)$.
   This is clearly not possible, as the equation $X^{q-1}=-1$ cannot have more than $q-1$ solutions.
 So $m_1^2+m_2^2=0$, which yields $m_2=\pm\nu m_1$ where
    $\nu^2=-1$. On the other hand, if $m_2=\pm\nu m_1$ and
    $q\equiv 1\pmod4$, then $m_1^{q+1}+m_2^{q+1}=m_1^{q+1}(1+\nu^{q+1})=0$,
    so~\eqref{eq4} is satisfied and the lines $\ell: y-\nu x =z=0$ and $\ell: y+\nu x =z=0$  are contained in $\cB_{\alpha,\beta}$.
    On the other hand, if $q\equiv3\pmod 4$, then $m_1^{q+1}+m_2^{q+1}=2m_1^{q+1}\neq 0$; so~\eqref{eq4} is not satisfied
    and there is no line contained in $\cB_{\alpha,\beta}$.

    Now, take $L_{\infty}=(0,a,b,c)\in \cB_{\infty} \setminus \{P_{\infty}\}$; hence $a^2+b^2=0$ and $a, b \neq 0$. Let  $r$ be a line through
$L_{\infty}$. We may assume that $r$ has (affine) parametric equations
 \[\left\{\begin{array}{l}
      x=l+at\\
      y=m+bt\\
     z=n+ct
      \end{array}\right.\]
where $t$ ranges over $\GF(q^2)$. Assume also that $(l,m,n)$ are the affine coordinates of a point in $\cB_{\alpha,\beta}$ that is,
\begin{equation}\label{add1}
n^q-n+\alpha^q(m^{2q}+l^{2q})-\alpha(m^2+l^2)=(\beta^q-\beta)(l^{q+1}+m^{q+1}).
\end{equation}
Now, $r$ is contained in $\cB_{\alpha,\beta}$ if and only if $q\equiv 1\pmod4$ and  the following condition holds:
\begin{equation}\label{fond}
c+2\alpha (al+bm)+(\beta-\beta^q)(al^q+bm^q)=0.
\end{equation}
Since $b=\nu a$ where  $\nu^2=-1$ and $\nu \in \GF(q)$,  setting $k=l+\nu m$  equation \eqref{fond} becomes
\begin{equation}\label{fond1}
c+2a\alpha k+a(\beta-\beta^q)k^q=0.
\end{equation}
From Lemma~\ref{tec}, the above equation has exactly one solution if and only if
\begin{equation}\label{int} (2\alpha)^{q+1}\neq(\beta-\beta^q)^{q+1}. \end{equation}
Considering that $2\in\GF(q)$ and $(\beta-\beta^q)^q=(\beta^q-\beta)$,
we obtain that \eqref{int} is equivalent to
\[ 4\alpha^{q+1}+(\beta^q-\beta)^2\neq 0, \]
which holds true.

Let $\bar{k}$ be the unique solution of \eqref{fond1}. Since $\bar{k}=l+\nu m$, we find $q^2$ pairs $(l,m)$  satisfying  \eqref{fond}. For any fixed pair $(l,m)$,
because of \eqref{add1}, there are
$q$ possible  values of $n$. Thus  we obtain that the number of affine lines  through the point $P$  and contained in $\cB_{\alpha,\beta}$ is $q^2q/q^2=q$.

Furthermore, the $q+1$ lines through $L_{\infty}$ lie  on the plane of (affine)
equation:
$x+\nu y=\bar{k}$. The theorem follows.
\end{proof}

 \begin{lemma}
   \label{rlem}
  If $q \equiv 1 \pmod 4$ then, for any point $P\in\cB_{\alpha,\beta}\setminus
   \cB_{\infty}$ there are two lines $r_1(P)$ and $r_2(P)$ through $P$ contained in
   $\cB_{\alpha,\beta}$ such that
   $r_i(P)\cap(\ell_i\setminus\{P_{\infty}\})\neq\emptyset$.
 \end{lemma}
 \begin{proof}
As we already know, $\cB_{\infty}$ is the union of two lines $\ell_1$ and $\ell_2$,  through the point $P_{\infty}$.
    By considering~\eqref{eq4}, we see that the point at infinity of
   the two lines through the origin $O \in \cB_{\alpha,\beta}$ are one on $\ell_1$ and one
   on $\ell_2$. The semilinear
   automorphism group $G$ of $\cB_{\alpha,\beta}$ is transitive
   on its affine points, see \cite{ACK} and maps lines into lines.
   Also, by Lemma~\ref{lm32} follows directly that $G$ must fix
   the hyperplane at infinity.
   If $q\equiv3\pmod4$, the point $P_{\infty}$ is the only point
   of $\cB_{\alpha,\beta}$ incident with two lines therein contained.
   So $P_{\infty}$ must be stabilized by $G$.
   If $q\equiv1\pmod4$, we see that $P_{\infty}$ is the only
   point at infinity incident with just $2$ lines of the variety,
   while
   the remaining points at infinity are incident with $q+1$ lines.
   So, again $P_{\infty}$, which is
   the only point of $\cB_{\infty}$
   which is on no affine line of $\cB_{\alpha,\beta}$, is fixed by $G$.
   It follows that
   for each affine point $P$ we have that one of the lines intersects
    with $\ell_1$ and the other with $\ell_2$.
 \end{proof}

\begin{theorem}
  \label{th32}
  Let $\cM_{\alpha,\beta}$ be the BM quasi-Hermitian variety described
  in~\eqref{qmv}.
 % Let $\ell_1$ and $\ell_2$ be the lines of equation $J=0=X$ and
 % $J=0=Y$.
  \begin{itemize}
  \item
    If  $q\equiv 1\pmod 4$ then through each affine point of $\cM_{\alpha,\beta}$ there pass two lines of $\cM_{\alpha,\beta}$ whereas through a point at infinity of $\cM_{\alpha,\beta}$  on the union of the two lines $\ell_1 \cup \ell_2$ there pass $q+1$ lines of a pencil contained in $\cM_{\alpha,\beta}$; finally through a point at infinity of $\cM_{\alpha,\beta}$ which is not on $\ell_1 \cup \ell_2$ there passes only one line of $\cM_{\alpha,\beta}$.
  \item
    If  $q\equiv 3\pmod 4$ then no line of $\cM_{\alpha,\beta}$ passes through any affine point of $\cM_{\alpha,\beta}$ whereas through a point at infinity of $(\cM_{\alpha,\beta}\cap \Sigma_{\infty}) \setminus P_{\infty}$  there passes only one line contained in $\cM_{\alpha,\beta}$.
    \item
      Through the point $P_{\infty}$ there are always $q+1$ lines contained in
      $\cM_{\alpha,\beta}$.
    \end{itemize}
\end{theorem}
\begin{proof}
We observe that the affine points of $\cM_{\alpha,\beta}$ are the same as those of $\cB_{\alpha,\beta}$, whereas the set $\cF$ of points at infinity of $\cM_{\alpha,\beta}$ consists of the points $P=(0,x,y,z)$ such that  $x^{q+1}+y^{q+1}=0$. Furthermore,  $\cB_{\infty}=\ell_1 \cup \ell_2$ is contained in $\cF$ if $q\equiv1\pmod4$.
Hence, from Lemma \ref{lm32} we get the result.
\end{proof}

 \section{Connected graphs from $\cM_{\alpha,\beta}$ in $\PG(3,q^2)$,  $q\equiv 1\pmod 4$  }
 \label{sec:3}
 Let $\cV$ be an algebraic variety in $\PG(n-1,q^2)$ or, more
 in general, just a set of points and suppose that
 $\cV$ contains some projective lines.
 Then we can define the \emph{collinearity graph} of $\cV$, say
 $\Gamma(\cV)=(\cP,\cE)$ as the graph whose vertices $\cP$
 are the points of $\cV$ and such that two points $P$ and $Q$ are collinear
 in $\Gamma(\cV)$ if and only if the line $[\langle P, Q\rangle]$ is
 contained in $\cV$.

 When $\cV$ is a (non-degenerate) quadric or Hermitian variety, the
 graph $\Gamma(\cV)$ has a very rich structure for it is strongly
 regular and admits a large automorphism group; this has been
 widely investigated; see~\cite[Chapter 2]{BvM22}, \cite{S11}.

 More in general,
 the properties of the graph $\Gamma(\cV)$ provide insight on
 the geometry of $\cV$ since any automorphism of $\cV$
 is also naturally an automorphism of $\Gamma(\cV)$, but the converse
 is not true in general.

 \begin{lemma}\label{diam}
   Let $\cV$ be an algebraic variety containing some lines and let $\cV_{\infty}=\cV\cap \Sigma_{\infty}$ where $\Sigma_{\infty}$ is an
   hyperplane of $\PG(n-1,q^2)$. If the graph $\Gamma({\cV_{\infty}})$   is connected and through each point of $\cV$ there passes at least one line of $\cV$ then the collinearity graph $\Gamma(\cV)$ is connected and its diameter $d(\Gamma({\cV}))$ is at most $d(\Gamma({\cV_{\infty}}))+2$.
 \end{lemma}
 \begin{proof}
   Each line of $\cV$ has at least  a point at infinity hence, given two points $P$ and $Q$ there exists a path from $P$ to a point at infinity $P'$ and from $Q$ to another point at infinity $Q'$ and finally a path
   consisting of points in $\cV_{\infty}$
   from $P'$ to $Q'$.
 \end{proof}
%Consider the variety  $\cB_{\alpha,\beta}$ in $PG(3,q^2)$, $q$ odd with equation \eqref{eqqh} such that $4\alpha^{q+1}+(\beta^q-\beta)^2\neq 0$.

Let $\cM_{\alpha,\beta}$ be as in~\eqref{qmv}.
 \begin{theorem}
  If $q\equiv 1 \pmod4$, then the graph $\Gamma(\cM_{\alpha,\beta})$ is connected and its diameter is $3$.
 \end{theorem}

 \begin{proof}
  We recall that  $\cB_{\alpha,\beta}\subseteq\cM_{\alpha,\beta}$, $\cB_{\alpha,\beta}\setminus\cB_{\infty}=\cM_{\alpha,\beta}\setminus \Sigma_{\infty}$ and that
   $\cB_{\infty}$
   splits in the
   union of the two distinct lines $\ell_1,\ell_2$ through  $P_{\infty}$.
   In particular, $\Gamma(\cB_{\infty})$ is a connected graph of diameter
   $2$.
   Take now two points $P,Q\in\cB_{\alpha,\beta}$.
   If $P,Q\in \cB_{\infty}$, then we have $d(P,Q)\leq 2$ and there is nothing
   to prove.
   Suppose now $P\in \cB_{\alpha,\beta}\setminus\cB_{\infty}$ and $Q\in\cB_{\infty}$.
   Suppose $Q\in\ell_i$. Then, from Lemma \ref{rlem} we can consider  a point
   $P'=r_i(P)\cap\ell_i$ where $r_i(P)$ is one of the two lines through $P$ which is contained in $\cB_{\alpha,\beta}$. If $P'=Q$, then $d(P,Q)=1$; otherwise
   $d(P,Q)=2$.

   Take now $P,Q\in\cB_{\alpha,\beta}\setminus\cB_{\infty}$. Then, again from Lemma \ref{rlem},   the lines
   $r_1(P)$ and $r_1(Q)$ meet $\ell_1$. Put $P'=r_1(P)\cap\ell_1$
   and $Q'=r_1(Q)\cap\ell_1$. If $P'=Q'$, then $d(P,Q)\leq 2$;
   otherwise $d(P,Q)\leq 3$.
   We now show that there are pairs of points in $\cM_{\alpha,\beta}$ which are at distance $3$.
   Take $P\in\cB_{\alpha,\beta}\setminus\cB_{\infty}$ and $Q\in\cM_{\alpha,\beta}\setminus\cB_{\alpha,\beta}$.
   Then, $Q$ is not collinear with any affine point by construction;
   also $Q$ is not collinear
   with $P_i:=\ell_i\cap r_i(P)$, $i=1,2$.
   So, the shortest paths from $P$ to $Q$ are of the form $P~P_i~P_{\infty}~Q$.
   It follows that $d(P,Q)=3$ and thus the diameter of the graph is $3$.

 \end{proof}

 \section{Main result}
\label{sec:4}
In this section
we show that the arguments of~\cite{BE92} for classifying BM unitals in
$\PG(2,q^2)$ can be extended to BM quasi-Hermitian varieties in
$\PG(3,q^2)$, $q$ odd. We keep all previous notations.

% As mentioned in the introduction~\eqref{eqqh} is not the
% equation of $\cM_{\alpha,\beta}$. The following lemma, however, shows
% that any collineation between $\cM_{\alpha,\beta}$ and $\cM_{\alpha',\beta'}$
% must map the corresponding affine equations of the varieties
% $\cB_{\alpha,\beta}$ in those of $\cB_{\alpha',\beta'}$.
Two BM quasi-Hermitian varieties $\cM_{\alpha,\beta}$ and $\cM_{\alpha',\beta'}$ of $\PG(3,q^2)$ are \emph{ equivalent} if there exists a semilinear collineation $\psi \in \PGamma(4,q^2)$ such that $\psi(\cM_{\alpha,\beta})=\cM_{\alpha',\beta'}$.
 \begin{lemma}
   \label{l51}
   Let $\psi$ be a semilinear collineation of $\PG(3,q^2)$, $q$ odd,  such that
   $\psi(\cM_{\alpha,\beta})=\cM_{\alpha',\beta'}$ where  $\cM_{\alpha,\beta}$ and $\cM_{\alpha',\beta'}$ are two BM quasi-Hermitian varieties.
   Then $\psi$ fixes $P_{\infty}$ and stabilizes $\Sigma_{\infty}$.
   Also, if $q\equiv1\pmod4$ then $\psi(\cB_{\alpha,\beta})=\cB_{\alpha',\beta'}$.
 \end{lemma}
 \begin{proof}
   First, we show that $\psi$ fixes $P_{\infty}$ for $q\equiv 3\pmod 4$.
   From Theorem \ref{th32} we have that $P_{\infty}$ is the only point of the two varieties contained in $q+1$ lines and hence $\psi(P_{\infty})= P_{\infty}$.
   Furthermore, we observe that $\Sigma_{\infty}$ is the  plane through $P_{\infty}$ meeting both $\cM_{\alpha,\beta}$ and  $\cM_{\alpha',\beta'}$ in $q^3+q^2+1$   points  which are on the $q+1$ lines through $P_\infty$.
All of the $q^3-q^2$ points of $\cM_{\alpha,\beta}$ and $\cM_{\alpha',\beta'}$
lying on exactly one line contained in the respective variety are
in this plane, and these points also span $\Sigma_{\infty}$.
So  also $\Sigma_{\infty}$ is left invariant by $\psi$.
\par
Now assume $q\equiv 1 \pmod 4$;
from Theorem \ref{th32}, for each point in $\ell_1 \cup \ell_2$ there pass $q+1$ lines  of the quasi-Hermitian varieties however
$P_{\infty}$ is the only point on $\ell_1\cup \ell_2$ such that the other  $q-1$ lines through it are not incident with other lines of the two varieties, hence we again obtain
$\psi(P_{\infty})= P_{\infty}$.
In this case $\cB_{\alpha,\beta}\subseteq\cM_{\alpha,\beta}$.
   Since $\psi(\Sigma_{\infty})=\Sigma_{\infty}$,
   we have
   \[ \psi(\cB_{\alpha,\beta}\setminus\Sigma_{\infty})=\psi(\cM_{\alpha,\beta}\setminus\Sigma_{\infty})=\cM_{\alpha',\beta'}\setminus\Sigma_{\infty}=\cB_{\alpha',\beta'}
     \setminus\Sigma_{\infty},\]
   that is $\psi$ stabilizes the affine part of $\cB_{\alpha,\beta}$.

Furthermore $\cB_{\infty}=\cB_{\alpha,\beta}\cap\Sigma_{\infty}$
   consists of the union of the two  lines, $\ell_1$ and $\ell_2$.
   Observe also that the lines through the affine points of
   $\cM_{\alpha,\beta}$ are also lines of $\cB_{\alpha,\beta}$ (see
   Theorem~\ref{th32}) and, in particular they are incident either
   $\ell_1$ or $\ell_2$. This is equivalent to say that the points of
   $\ell_1\cup\ell_2$ different from $P_{\infty}$
   are exactly the points of $\Sigma_{\infty}$ through which
   there pass some affine lines of $\cM_{\alpha,\beta}$.
   This implies that $\psi(\ell_1\cup\ell_2)=\ell_1\cup\ell_2$ and,
   consequently
   \[
     \psi(\cB_{\alpha,\beta})=\psi(\cB_{\alpha,\beta}\setminus\Sigma_{\infty})
     \cup\psi(\ell_1\cup\ell_2)=(\cM_{\alpha',\beta'}\setminus\Sigma_{\infty})
     \cup(\ell_1\cup\ell_2)=\cB_{\alpha',\beta'}.
   \]
 \end{proof}

 \begin{theorem}
   \label{lorb}
   Suppose $q\equiv1\pmod4$. Let  $\cG$ be the group of collineations
   $\cG=\mathrm{Aut}(\cM_{\alpha,\beta})\subseteq\mathrm{P\Gamma{}L}(4,q^2)$ and
   $\fG$
   the group of graph automorphisms
   $\fG=
   \mathrm{Aut}(\Gamma(\cM_{\alpha,\beta}))$.
   Then the sets
   \begin{itemize}
   \item $\Omega_0:=\{P_{\infty}\}$;
   \item $\Omega_1$ consisting of
   the points at infinity of $\cB_{\alpha,\beta}$ different from
   $P_{\infty}$;
 \item $\Omega_2:=\cM_{\alpha,\beta}\setminus\Sigma_{\infty}$
 \end{itemize}
 are all stabilized by both $\cG$ and $\fG$.
 Furthermore, $\Omega_3=\cM_{\alpha,\beta}\setminus\cB_{\alpha,\beta}$ is
 an orbit for $\fG$.
 \end{theorem}
 \begin{proof}
   By~\cite[\S 4]{ACK}, we know that there is a subgroup of $\cG$
   which is transitive on the affine points of
   $\cM_{\alpha,\beta}$ i.e. on $\Omega_2$.
   By Lemma~\ref{l51}, any collineation in $\cG$ must
   stabilize the plane $\Sigma_{\infty}$; so any element of $\cG$ maps
   points of $\Omega_2$ into points of $\Omega_2$ and $\Omega_2$ is
   an orbit of $\cG$.
   Also by Lemma~\ref{l51},  $\Omega_0:=\{P_{\infty}\}$
   is fixed by any  $\gamma\in\cG$.
   So we have that the points at infinity of
   $\cB_{\alpha,\beta}\setminus\{P_{\infty}\}$, as well as the points of
   $\cM_{\alpha,\beta}\setminus\cB_{\alpha,\beta}$, are union of orbits.
   Let $\ell_1,\ell_2$ be the two lines of $\cB_{\alpha,\beta}$ at infinity.
   Using Lemma~\ref{rlem}, we see that $\cG$ is transitive on
   $\Omega_1=(\ell_1\cup\ell_2)\setminus\{P_{\infty}\}$.
   Indeed, for any two points $P,Q\in\ell_1\setminus\{P_{\infty}\}$,
   by Lemma~\ref{lm32},
   there are points $P_0, Q_0\in\Omega_2$ such that $r_1(P_0)\cap\Sigma_{\infty}=\{P\}$ and
   $r_1(Q_0)\cap\Sigma_{\infty}=\{Q\}$.

   Since $\cG$ is transitive on $\Omega_2$, there is
   $\gamma\in\cG$ such that $\gamma(P_0)=Q_0$. It follows that
   $\gamma ((r_2(P_0)\cap\Sigma_{\infty})\cup\{P\})=
   (r_2(Q_0)\cap\Sigma_{\infty})\cup\{Q\}$.
   If $\gamma(P)=Q$, then we are done.
   Otherwise, consider
   the element $\theta:(J,X,Y,Z)\to (J,X,-Y,Z)$ of $\cG$.
   Observe that $\theta(r_2(Q_0))\cap\Sigma_{\infty}=
   r_1(Q_0)\cap\Sigma_{\infty}$. Hence,
   $\theta\gamma(P)=Q$.
   Also, $\theta(\ell_1)=\ell_2$; so it follows that
   $\Omega_1:=(\ell_1\cup\ell_2)\setminus\{P_{\infty}\}$
   is an orbit of $\cG$.

 Since $\fG$ contains $\cG$, the orbits of $\fG$ are possibly unions
 of orbits of $\cG$. However, observe that the points of
 $\Omega_3$ are the only points of $\cM_{\alpha,\beta}$ which are
 on exactly one line of $\cM_{\alpha,\beta}$ through the point $P_{\infty}$.
 So these points must be permuted among each other also by $\fG$.

 The same argument shows that $\Omega_0$ is
 also an orbit for $\fG$.
 Now, consider the points of $\Omega_2$. They are the points of
 $\cB_{\alpha,\beta}\setminus\Omega_0$ incident with exactly $2$ lines,
 while the points of $\Omega_1$ are incident with more than $2$ lines.
 So $\fG$ cannot map a vertex in $\Omega_2$ into a vertex in $\Omega_1$
 and these orbits are distinct.

 Put $\Gamma:=\Gamma(\cM_{\alpha,\beta})$.
 Observe that the graph $\Gamma\setminus\{P_{\infty}\}$
is the disjoint union of $\Gamma(\Omega_3)$ and
$\Gamma(\Omega_1\cup\Omega_2)$.
In turn,  $\Gamma(\Omega_3)$ consists of the disjoint union $K_1\cup K_2\cup\dots\cup K_{q-1}$ of $q-1$ copies of the complete graph on $q^2$ elements.
Write $\{ v_{i}^j\}_{j=1,\dots,q^2}$ for the list of vertices of $K_i$ with
$i=1,\dots,q-1$.

Also, each vertex of $\Gamma(\Omega_3\cup\{P_{\infty}\})$
is collinear with $P_{\infty}$.
Let $S_{q^2}$ be the symmetric group on $q^2$ elements, and consider its
action on $\Gamma$ given by
\[ \forall\xi\in S_{q^2}: \check{\xi}(v_1^j):=v_{1}^{\xi(j)}\]
if $v_1^j\in K_1$ and fixing all remaining vertices.
Obviously $\check{S}_{q^2}<\fG$ and $\check{S}_{q^2}$ is transitive on $K_1$.
Let $S_{q-1}$ be the symmetric group on $\{1,\dots,q-1\}$ and consider
its action  on $\Gamma$ given by
\[ \forall \sigma\in S_{q-1}:
  \hat{\sigma}(v_i^j):=v_{\sigma(i)}^{j},\quad j=1,\dots,q^2 \]
and all the remaining vertices of $\Gamma$ are fixed.
We also have $\hat{S}_{q-1}<\fG$ and $\hat{S}_{q-1}$ permutes the
sets $K_i$ for $i=1,\dots,q-1$.
By construction, we see that
the wreath product
$\check{S}_{q}\wr \hat{S}_{q-1}$ is a subgroup of $\fG$,
it acts naturally on $\Gamma$, fixes all vertices not
in $\Omega_3$ and acts  transitively on $\Omega_3$.
It follows that $\fG$ is transitive on $\Omega_3$.
 \end{proof}

\begin{remark}
  It can be easily seen that the
  automorphism group of $\Gamma:=\Gamma(\cM_{\alpha,\beta})$ is in
  general much larger than the subgroup of collineations
  stabilizing $\cM_{\alpha,\beta}$. In particular the elements
  of $\check{S}_q\wr\hat{S}_{q-1}$ are not, in general, collineations.
  For instance, in the case $q=5$ with $\alpha=\beta=\varepsilon$
  where $\varepsilon$ is a primitive element of $\GF(25)$,
  root of $x^2-x+2$ in $\GF(5)$,
  the group $\cG$ has order
  $2^65^5$, while
  $\fG$ has order
  $2^{99}3^{42}5^{30}7^{12}11^{8}13^417^419^423^4$.
  In this case also $\cG$ is transitive on $\Omega_3$.
\end{remark}

 \begin{lemma}
   \label{col-lemma}
If $\cM_{\alpha,\beta}$ and $\cM_{\alpha',\beta'}$ are two equivalent BM quasi-Hermitian varieties then  there is a semilinear collineation $\phi : \cM_{\alpha,\beta}\rightarrow \cM_{\alpha',\beta'}$   of the following type
\[\phi(j,x,y, z) =(j^{\sigma},x^{\sigma},y^{\sigma},z^{\sigma})M,\  where \]
\[M=\begin{pmatrix}
   a&0&0&0\\
   0&b&c&0\\
0& c& -b&0\\
0&0&0&1
  \end{pmatrix}, \  or  \ M=\begin{pmatrix}
   a&0&0&0\\
   0&b&c&0\\
0& -c& b&0\\
0&0&0&1
  \end{pmatrix},\]
$\sigma \in\mathrm{Aut}(\GF(q^2))$, $a\in \GF(q)\setminus \{0\}$, $b,c\in \GF(q^2)$,  $b^2+c^2\neq 0$ and if $b\neq 0 \neq c$ then $c=\lambda b$ with $\lambda \in \GF(q)\setminus \{0\}$ such that $\lambda^2+1\neq0$.

\end{lemma}
\begin{proof}
  By Lemma~\ref{l51}, $\phi$ fixes the point $P_{\infty}$ and stabilizes
  $\Sigma_{\infty}$. As the automorphism group of $\cM_{\alpha,\beta}$ is
  transitive on its affine points,
  we can also assume that $\phi(1,0,0,0)=(1,0,0,0)$.
  More in detail, let $G'$ be the collineation group of $\cM_{\alpha',\beta'}$   fixing $P_{\infty}$, leaving $\cF \setminus P_{\infty}$ invariant  and transitive on the affine points of  $\cM_{\alpha',\beta'}$. If $\phi(1,0,0,0)\neq (1,0,0,0)$ we can consider the collineation $\phi' \in G'$ mapping $\phi(1,0,0,0)$ to $(1,0,0,0)$ and then we replace $\phi$ by $\phi \phi'$. This implies that
 $\phi$ has the following form up to scalar multiple
\[\phi(j,x,y, z) =(j^{\sigma},x^{\sigma},y^{\sigma},z^{\sigma})\begin{pmatrix}
   a&0&0&0\\
   0&b&c&d\\
0& e& f&g\\
0&0&0&1
  \end{pmatrix}, \]
where $\sigma \in\mathrm{Aut}(\GF(q^2))$, $a,b,c,d,e,f,g \in \GF(q^2)$ and $a \neq 0 \neq bf-ce$.

Since $(1,0,0,c)$, belongs to $\cM_{\alpha, \beta}$ if
and only if $c\in\GF(q)$, it follows that
$\phi(1,0,0,c)=(a,0,0,c)\in \cM_{\alpha', \beta'}$ implies $ca^{-1}\in
\GF(q)$, and thus $a\in\GF(q)^*$.
Now we observe that the  affine plane $Y=0$ has in common with $\cM_{\alpha,\beta}$ the points $(1,x,0,z)$ for which
$-\alpha x^2+\beta x^{q+1}-z \in \GF(q)$; so, $a^{-1}(-\alpha^{\sigma} x^{2\sigma}+\beta^{\sigma}x^{\sigma(q+1)}-z^{\sigma})\in \GF(q)$.
Thus, suppose that $(1,x,0,z)\in \cM_{\alpha,\beta}$; we have  $\phi(1,x,0,z) \in \cM_{\alpha', \beta'}$  and therefore
\begin{equation}\label{col1}
(\alpha^{\sigma}-\alpha'(b^2+c^2)/a)x^{2\sigma}-(\beta^{\sigma}-\beta'(b^{q+1}+c^{q+1})/a)x^{\sigma(q+1)}-dx^{\sigma}\in \GF(q),
\end{equation}
as $\sigma$ stabilizes $\GF(q)$.
Let $\eta \in \GF(q^2)\setminus \GF(q) $ such that $\eta^2$ is a primitive element of $\GF(q)$.
Considering $x^{\sigma}=1,-1,\eta, -\eta, 1+\eta$ in \eqref{col1}, we get
\[d=0,\]
\begin{equation}\label{add2}
\alpha^{\sigma}-\alpha'(b^2+c^2)/a=0,\end{equation}
%\begin{equation}\label{add3}
\[
  \beta^{\sigma}-\beta'(b^{q+1}+c^{q+1})/a\in \GF(q).
\]
%\end{equation}
Similarly if we consider the affine points in common between the plane $X=0$ and $\cM_{\alpha,\beta}$, arguing as before, we obtain \[g=0\]
% \begin{equation}\label{eqq1}
\[
  \alpha^{\sigma}-\alpha'(e^2+f^2)/a=0,
  \]
%\end{equation}
\begin{equation}\label{eqq2}
\beta^{\sigma}-\beta'(e^{q+1}+f^{q+1})/a\in \GF(q).\end{equation}
In particular,
\begin{equation}\label{eqq3}
b^2+c^2=e^2+f^2\neq 0.
\end{equation}
Also, since $\beta'\not\in\GF(q)$,
\begin{equation}\label{eqq4}
b^{q+1}+c^{q+1}=e^{q+1}+f^{q+1}\neq 0.
\end{equation}
Now we recall that a generic point
$(1,x,y,z)\in \cM_{\alpha,\beta}$ if and only if $\phi(1,x,y,z)\in \cM_{\alpha',\beta'}$.
On the other hand,
\[(1,x,y,z)\in \cM_{\alpha,\beta} \Leftrightarrow -\alpha(x^2+y^2)+\beta(x^{q+1}+y^{q+1})-z\in \GF(q).\]
Since $a \in\GF(q)\setminus\{0\}$ and $\sigma$ stabilizes $\GF(q)$,
the former equation is equivalent to
\begin{equation}\label{res}
a^{-1}\{-\alpha^{\sigma} (x^{2\sigma}+y^{2\sigma})+\beta^{\sigma}[x^{\sigma(q+1)}+y^{\sigma(q+1)}]-z^{\sigma}\}\in \GF(q).
\end{equation}
Next, we observe that
$\phi(1,x,y,z)=(1,\frac{bx^{\sigma}+ey^{\sigma}}{a}, \frac{cx^{\sigma}+fy^{\sigma}}{a}, \frac{z^{\sigma}}{a} ) $ and this point belongs to $\cM_{\alpha', \beta'}$    if and only if
\begin{multline}\label{res1}
a^{-1}\Big\{-\alpha' \left[\frac{(bx^{\sigma}+ey^{\sigma})^2}{a}+ \frac{(cx^{\sigma}+fy^{\sigma})^2}{a}\right]+ \\ \beta'\left[ \frac{(bx^{\sigma}+ey^{\sigma})^{(q+1)}}{a}+\frac{(cx^{\sigma}+fy^{\sigma})^{(q+1)}}{a}\right]-{z^{\sigma}}\Big\} \in \GF(q)
\end{multline}
From \eqref{res} and \eqref{res1},
we get that for all $(1,x,y,z)\in \cM_{\alpha,\beta}$ the following holds:
\begin{multline*}
\alpha^{\sigma} (x^{2\sigma}+y^{2\sigma})-\alpha' \left[\frac{(bx^{\sigma}+ey^{\sigma})^2}{a}+\frac{(cx^{\sigma}+fy^{\sigma})^2}{a}\right]+\\
+\beta'\left[ \frac{(bx^{\sigma}+ey^{\sigma})^{(q+1)}}{a}+\frac{(cx^{\sigma}+fy^{\sigma})^{(q+1)}}{a}\right]-\beta^{\sigma}[x^{\sigma(q+1)}+y^{\sigma(q+1)}] \in \GF(q)
\end{multline*}
that is,  using~\eqref{eqq3} and
\eqref{eqq4},
\begin{equation}\label{res3}
  -\alpha'
  %a^{-1}
  \left[ 2 x^{\sigma}y^{\sigma}(be+cf) \right]+ \beta' \left[(b^qe+c^qf) x^{\sigma q}y^{\sigma}+(be^q+cf^q) x^{\sigma}y^{\sigma q} \right]\in \GF(q)
\end{equation}
We are going to prove that $b^qe+c^qf=0$. Thus, let $\nu \in\GF(q^2)$ be any solution of  $X^{q+1}=-1$.
The semilinear collineation $\phi$ has to leave invariant the Hermitian cone $\cF $ that  is , $\phi(0,x,\nu x,z)\in \cF$,
 and because of the first equation in \eqref{eqq4} this means
 % \begin{equation}\label{res4}
 \[
(b^qe+c^qf)\nu^{\sigma}+(be^q+cf^q)\nu^{\sigma q}=0
\]
%\end{equation}
for any of the $q+1$ different solutions of $X^{q+1}=-1$.
If $(b^qe+c^qf)\neq 0$ then  the equation $ (b^qe+c^qf)X+(b^qe+c^qf)^q X^{q}=0$ would have more than $q$ solutions which is impossible.
Thus,
\begin{equation}\label{add4}
b^qe+c^qf=0
\end{equation}
and since $\alpha' \notin \GF(q)$ \eqref{res3} gives
\begin{equation}\label{eqq5}
be+cf=0.
\end{equation}
%In particular, if
Since $\det(M)\neq 0$, it cannot be $ce=0=bf$, so
either $c\neq 0 \neq e$ or $b\neq 0 \neq f$. Thus, from \eqref{eqq3} and \eqref{eqq5} we also get
$(e,f)=(c,-b)$ or $(e,f)=(-c,b)$. Thus from \eqref{add4} we also obtain
\begin{equation}\label{res5}
b^qc-bc^q=0.
\end{equation}
Hence if $b\neq 0 \neq c$ then $c=\lambda b$ where $\lambda \in \GF(q)$ and $\lambda^2+1\neq 0$. So the lemma follows.

\end{proof}
From the previous Lemma, taking into account  conditions  from \eqref{add2} to \eqref{eqq2}, we get that if  $\cM_{\alpha,\beta}$ and $\cM_{\alpha',\beta'}$ are equivalent, then
\begin{equation}\label{equiv}(
  \alpha',\beta')=(a\alpha^{\sigma}/(b^2+c^2), a\beta^{\sigma}/(b^{q+1}+c^{q+1})+u)\end{equation} for some   $\sigma \in\mathrm{Aut}(\GF(q^2))$, $a\in\GF(q)^*$, $u\in \GF(q)$, $b,c \in \GF(q^2): b^2+c^2\neq 0$ and if $b\neq 0\neq c$ then $c=\lambda b$ with $\lambda \in \GF(q)\setminus \{0\}$. Conversely, if condition~\eqref{equiv} holds,
there is a semilinear collineation $\cM_{\alpha,\beta}\to\cM_{\alpha',\beta'}$;
so $\cM_{\alpha,\beta}$ and $\cM_{\alpha',\beta'}$ are  equivalent.

In this case we write $(\alpha, \beta) \sim (\alpha',\beta')$ where $\sim$ is in particular an equivalence relation on the ordered pairs $(\alpha,\beta)\in \GF(q^2)^2$ such that $4\alpha^{q+1}+(\beta^q-\beta)^2 \neq 0$.

\begin{lemma}\label{lemadd5}
  Let $\cM_{\alpha,\beta}$ be a BM quasi-Hermitian variety of $\PG(3,q^2)$, $q$ odd
  and $\varepsilon$ be a primitive element of $\GF(q^2)$.
  Then, there exists $\alpha'\in\GF(q^2)\setminus\{0\}$ such that
  $\cM_{\alpha,\beta}$ is  equivalent to $\cM_{\alpha',\varepsilon}$.
\end{lemma}
\begin{proof}
Write $\beta=\beta_0+\varepsilon \beta_1$, with $\beta_0,\beta_1 \in \GF(q)$ and $\beta_1 \neq 0$. Then, there exists $b\in \GF(q^2)\setminus \{0\}$, such that  $\beta_1/b^{q+1}=1$.
Therefore $(\alpha,\beta) \sim (\alpha/b^2, \beta/b^{q+1}-\beta_0/b^{q+1})=(\alpha/b^2, \varepsilon) $.
\end{proof}
In light of the previous lemma, in order to determine the equivalence
classes of BM quasi-Hermitian varieties it is enough to determine when
two varieties $\cM_{\alpha,\varepsilon}$ and $\cM_{\alpha',\varepsilon}$
are equivalent. This is done in the following.

\begin{lemma}
  \label{main-lemma}
  Let $q=p^n$ be an odd prime,  $\varepsilon$ be a primitive element of $\GF(q^2)$, $\cM_{\alpha,\varepsilon}$
  and $\cM_{\alpha',\varepsilon}$ be two BM quasi-Hermitian varieties
  of $\PG(3,q^2)$.
  Put
  \[
    \delta(\alpha):=\frac{(\varepsilon^q-\varepsilon)^2}{4\alpha^{q+1}}.
  \]
  Then, $\cM_{\alpha,\varepsilon}$ is equivalent to
  $\cM_{\alpha',\varepsilon}$ if and only if there exist $\sigma\in\mathrm{Aut}(\GF(q^2))$ such that
\[ \delta(\alpha')=\delta(\alpha)^{\sigma}.
\]
\end{lemma}
\begin{proof}
First we observe that for all $\alpha \in\GF(q^2)\setminus \{0\}$ such that $4\alpha^{q+1}+(\varepsilon^q-\varepsilon)^2 \neq 0$   \[\delta(\alpha):=\frac{(\varepsilon^q-\varepsilon)^2}{4\alpha^{q+1}}\] belongs to $ \GF(q)\setminus\{0,-1\}$.
Conversely, given any $\delta \in\GF(q)\setminus \{0,-1\}$  we can generate some BM quasi-Hermitian varieties $\cM_{\alpha, \varepsilon}$, by choosing $\alpha$ to
be any solution of $4 \delta x^{q+1}=(\varepsilon^q-\varepsilon)^2$.      In fact, it turns out that $(\varepsilon ^q-\varepsilon)^2+4\alpha^{q+1}\neq 0$. Furthermore, let $\alpha_1$ and $\alpha_2$ be any two such solutions. Then there
exists $k$ such that $\alpha_2=\varepsilon^{k(q-1)}\alpha_1$.
On the other hand,
$(\alpha_1, \varepsilon) \sim (\alpha_1 \varepsilon^{-2} \varepsilon^{q+1}, \varepsilon \varepsilon^{-(q+1)}\varepsilon^{q+1})=(\alpha_1\varepsilon^{q-1}, \varepsilon)$.
By repeating this process $k$ times, we see
%\begin{equation}
%  \label{eqk}
\[
  (\alpha_1,\varepsilon) \sim
(\alpha_1\varepsilon^{k(q-1)},\varepsilon)=(\alpha_2, \varepsilon).
\]
%\end{equation}
Thus $\delta(\alpha_1)=\delta(\alpha_2)$ implies that
$\cM_{\alpha_1,\varepsilon}$ is  equivalent to
$\cM_{\alpha_2,\varepsilon}$.
Hence, in order to determine the number $N$ of inequivalent BM quasi-Hermitian varieties we need to count the number of "inequivalent"
$\delta \in \GF(q)\setminus\{0,-1\}$.
% We distinguish the following two cases :
% \begin{itemize}
% \item Case (A) $4\alpha^{q+1}+(\beta^q-\beta)^2 $ a non-square of $\GF(q)$.
% \item Case (B) $4\alpha^{q+1}+(\beta^q-\beta)^2 $ a non-zero square of $\GF(q)$.
% \end{itemize}

Now, given two BM quasi-Hermitian varieties $\cM_{\alpha,\varepsilon}$ and $\cM_{\alpha',\varepsilon}$ and setting
\[ \delta=\delta(\alpha)=\frac{(\varepsilon^q-\varepsilon)^2}{4\alpha^{q+1}},\quad
\delta'=\delta(\alpha')=\frac{(\varepsilon^q-\varepsilon)^2}{4{\alpha'}^{q+1}}, \]
we have to show that $\cM_{\alpha,\varepsilon}\sim\cM_{\alpha',\varepsilon}$
if and only if $\delta'=\delta^{\sigma}$ for some $\sigma\in\mathrm{Aut}(\GF(q^2))$.

First, suppose that $\cM_{\alpha,\varepsilon}$ and
$\cM_{\alpha',\varepsilon}$
 are equivalent that is, $(\alpha',\varepsilon) \sim (\alpha,\varepsilon)$.
This is true if and only if
\[ \alpha'=\frac{\alpha^{\sigma}a}{b^2+c^2},
  \quad
  \varepsilon=\frac{a\varepsilon^{\sigma}}{b^{q+1}+c^{q+1}}+u, \]
for some $\sigma \in \mathrm{Aut}(\GF(q^2))$, $a \in \GF(q)\setminus \{0\}$, $u\in \GF(q)$, $b,c\in \GF(q^2)$ such that the conditions in the thesis of Lemma~\ref{col-lemma} hold.

%We already know that if $\delta'=\delta$, then $\cM_{\alpha,\varepsilon}
%\sim\cM_{\alpha',\varepsilon}$.

Then
\[ \delta'=(b^2+c^2)^{q+1}\frac{(\varepsilon^q-\varepsilon)^2}{4a^2(\alpha^{\sigma})^{q+1}},\quad
  \delta^{\sigma}=(b^{q+1}+c^{q+1})^2\frac{(\varepsilon^q-\varepsilon)^2}{4a^2(\alpha^{\sigma})^{q+1}}.\]

  We observe that
\begin{equation}
  \label{dpds}
  (b^2+c^2)^{q+1}=(b^{q+1}+c^{q+1})^2.
\end{equation}
In fact, if either $b=0$ or $c=0$, then~\eqref{dpds} is trivially satisfied
and there is nothing further
to prove.
Otherwise,  a direct manipulation yields that \eqref{dpds} is equivalent to
\[ \frac{b^{q-1}}{c^{q-1}}+\frac{c^{q-1}}{b^{q-1}}=2.
\]
This gives $\frac{b^{q-1}}{c^{q-1}}=1$, which is always true, since
\eqref{res5} holds.
Because of \eqref{dpds} then $\delta'=\delta^{\sigma}$.

Conversely, suppose that  $\delta'=\delta^{\sigma}$ for some $\sigma$.
Then we observe that  $(\alpha,\varepsilon)\sim (\alpha^{\sigma}, \varepsilon^{\sigma}) $. Furthermore
$(\alpha^{\sigma}, \varepsilon^{\sigma})\sim (\alpha^{\sigma}/b^2,\varepsilon)$ where $\varepsilon^{\sigma}=b_1\varepsilon+b_0$ with  $b_1/b^{q+1}=1$ for a suitable $b\in \GF(q^2)\setminus\{0\}$, as seen in the proof of Lemma \ref{lemadd5}.

Thus we have that
\begin{multline*}
\delta(\alpha^{\sigma}/b^2) =(\varepsilon^q-\varepsilon)^2 (b^2)^{q+1}/4(\alpha^{\sigma})^{q+1}=\\
(b^2)^{q+1}\left\{[(\varepsilon^{\sigma})^q-b_0]-(\varepsilon^{\sigma}-b_0)\right\}^2/(4(\alpha^{\sigma})^{q+1}(b^{q+1})^2)=\\
[(\varepsilon^q-\varepsilon)^2]^{\sigma}/4(\alpha)^{(q+1)\sigma}=\delta^{\sigma}=\delta'.
\end{multline*}
Hence,
\[(\alpha',\varepsilon)\sim (\alpha^{\sigma}/b^2,\varepsilon)\sim (\alpha^{\sigma},\varepsilon^{\sigma})\sim(\alpha, \varepsilon)\]

\end{proof}
\begin{conjecture}
  We conjecture that Lemma~\ref{main-lemma} holds for all odd $r\geq 3$,
  as the conditions on the coefficients $\alpha,\beta$ are the same
  and the block structure of the matrices
  representing the classes should be analogous to
  that of Lemma~\ref{col-lemma}.
  For $r$ even the algebraic conditions on $\alpha$ and $\beta$  to construct quasi-Hermitian varieties are different, see \cite{ACK}.
\end{conjecture}

\begin{theorem}
  Let $q=p^n$ with $p$ an odd prime. Then
the number $N$ of inequivalent BM quasi-Hermitian varieties $\cM_{\alpha,\beta}$ of $\PG(3,q^2)$ is
\[N=\frac{1}{n}\left(\sum_{k|n}\Phi\left(\frac{n}{k}\right)p^k\right)-2,\]
where $\Phi$ is the Euler $\Phi$-function.
\end{theorem}

\begin{proof}
For all $\delta, \delta' \in \GF(q) \setminus \{0,-1\}$ write
$\delta \sim \delta'$ if and only if $\delta'=\delta^{\sigma}$ for some $\sigma \in \mathrm{Aut}(\GF(q^2))$. By Lemma~\ref{main-lemma}, $N$ is the number of inequivalent classes $[\delta]$  under $\sim$.
Let $N_e=|\{ \delta \in \GF(p^e) \setminus \{0,-1\}:$ $ \delta$ is not contained in any smaller subfield of $\GF(q)\}|$. We have
\[ N=\sum_{e|n}\frac{N_e}{e}. \]
Observing that
\[\sum_{e'|e}N_{e'}=p^e-2,\]
denote by $\mu(x)$ the M\"obius function.
Then,
M\"obius inversion gives
 \[N_e=\sum_{e'|e}\mu(e')p^{e/e'}-2\sum_{e'|e}\mu(e').\]
It follows that \[N=(\sum_{e|n}\frac{1}{e}\sum_{e'|e}\mu(e')p^{e/e'})-2.\]
Let $m=e/e'$ be a divisor of $n$, then the coefficient of $p^m$ is
\[ \frac{1}{n}\sum_{(e/m)|(n/m)}\mu\left(\frac{e}{m}\right)\frac{n/m}{e/m}=\frac{1}{n} \Phi\left(\frac{n}{m}\right)\]
and finally \[N=\frac{1}{n}\left(\sum_{k/n}\Phi(\frac{n}{k})p^{k}\right)-2.\]

\end{proof}
%Suppose that $\phi(P_{\infty}) \neq  P_{\infty}$.
%Then the collineation group $\Gamma=<G', \phi^{-1} G \phi>$ where $G$ ($G'$) is the collineation group of $\cM_{\alpha,\beta}$   fixing $P_{\infty}$, with three orbits on $\cF$  %and transitive on the affine points of  $\cM_{\alpha,\beta}$.

\end{document}